\documentclass[12pt]{amsart}
\usepackage{amscd,amssymb,longtable, rotating, lscape, graphicx}
\usepackage[matrix,arrow,curve]{xy}
\usepackage{supertabular}
\usepackage{setspace}
\usepackage{tikz}
\usetikzlibrary{calc}
\usepackage{tikz-cd}
\usepackage{multirow}
\usepackage{graphicx}
\usepackage{subfigure}
\usepackage{mathtools}
\allowdisplaybreaks

\theoremstyle{definition}
\newtheorem{theorem}[equation]{Theorem}
\newtheorem*{theoremE*}{Theorem E}
\newtheorem*{theoremD*}{Theorem D}
\newtheorem*{theoremC*}{Theorem C}
\newtheorem*{theoremB*}{Theorem B}
\newtheorem*{theoremA*}{Theorem A}
\newtheorem*{theoremMAIN*}{Main Theorem}

\newtheorem{proposition}[equation]{Proposition}
\newtheorem*{definition*}{Definition}

\theoremstyle{remark}

\makeatletter\@addtoreset{equation}{section}
\makeatother

\makeatletter\@addtoreset{equation}{section}
\makeatletter\@addtoreset{section}{part}

\makeatother

\usepackage{xcolor}

\usepackage[symbol]{footmisc}

\author{Dasol Jeong}
\author{In-Kyun Kim}
\author{Jihun Park}
\author{Joonyeong Won}

\title{Sasaki-Einstein metrics on connected sums of $S^3\times S^4$}

\begin{document}

\begin{abstract}
We completely characterize the $n$-fold connected sums of  $S^3\times S^4$  that admit Sasaki--Einstein structures by proving that such a structure exists if and only if $n$ is even.
\end{abstract}


\address{ \emph{Dasol Jeong}\newline \textnormal{Department of
Mathematics, POSTECH
\newline \medskip 77 Cheongam-ro, Nam-gu, Pohang, Gyeongbuk, 37673, Korea \newline
Center for Geometry and Physics, Institute for Basic Science
\newline
79 Jigok-ro127beon-gil,, Nam-gu, Pohang, Gyeongbuk, 37673, Korea \newline
\texttt{jdsleader@postech.ac.kr}}}

\address{ \emph{In-Kyun Kim}\newline \textnormal{
June E. Huh Center for Mathematical Challenges, KIAS\newline
85 Hoegiro, Dongdaemun-gu, Seoul 02455, Korea\newline
\texttt{soulcraw@gmail.com}}}

\address{ \emph{Jihun Park}\newline \textnormal{Center for Geometry and Physics, Institute for Basic Science
\newline \medskip 79 Jigok-ro127beon-gil, Nam-gu, Pohang, Gyeongbuk, 37673, Korea \newline
Department of
Mathematics, POSTECH
\newline
77 Cheongam-ro, Nam-gu, Pohang, Gyeongbuk,  37673, Korea \newline
\texttt{wlog@postech.ac.kr}}}

\address{ \emph{Joonyeong Won}\newline \textnormal{Department of Mathematics, Ewha Womans University\newline
52 Ewhayeodae-gil, Seodaemun-gu, Seoul,
03760, Korea. \newline
\texttt{leonwon@ewha.ac.kr}}}

\maketitle

\section{Introduction}

A Riemannian manifold $(M, g)$ is called Sasakian  if the cone metric $r^2g+ dr^2$ defines a K\"ahler metric on $M\times \mathbb{R}^+$.  
A Sasakian  manifold is called Sasaki-Einstein if its metric $g$ satisfies the Einstein condition, i.e., $\mathrm{Ric}_g=\lambda g$ for some constant $\lambda$. 
A closed simply connected manifold that accommodates a Sasaki-Einstein structure must be spin (\cite[Proposition~4.2]{Mo97}).

Closed simply connected $5$-manifolds were completely classified by Barden and Smale (\cite{Ba65}, \cite{SM62}). In particular, a closed simply connected spin $5$-manifold $M$ is determined, up to diffeomorphism, by its second integral homology group $\mathrm{H}_2(M,\mathbb{Z})$. If $\mathrm{H}_2(M,\mathbb{Z})$ is torsion-free with rank $n$, then $M$ is diffeomorphic to the $n$-fold connected sum
\[
(S^2\times S^3)^{\#n}.
\]

It is known that every such manifold admits Sasaki--Einstein structures (\cite{Ko07}).

\begin{theorem}\label{theorem:Kollar}
For every nonnegative integer $n$, the manifold
\[
(S^2\times S^3)^{\#n}
\]
admits a Sasaki--Einstein structure.
\end{theorem}

A comprehensive classification of closed smooth $2$-connected $7$-manifolds was established by Crowley and Nordström, building on Crowley’s earlier classification up to almost diffeomorphism (\cite{Cr02,CrN19}). In particular, closed $2$-connected $7$-manifolds with torsion-free third integral homology groups are classified, up to almost diffeomorphism, by their third Betti numbers and first Pontryagin classes. If a closed oriented manifold bounds a parallelizable manifold, then all of its Pontryagin classes vanish. Consequently, a closed oriented $2$-connected $7$-manifold with torsion-free third integral homology group that bounds a parallelizable $8$-manifold is determined, up to almost diffeomorphism, solely by its third Betti number $b_3$. Moreover, such a manifold is diffeomorphic to
\[
 \Sigma\#(S^3\times S^4)^{\#b_3},
\]
where $\Sigma$ is one of the $28$ smooth homotopy $7$-spheres.

As extensions of Theorem~\ref{theorem:Kollar}, one may ask the following questions:
\begin{enumerate}

\item Is there a Sasaki-Einstein $7$-manifold homeomorphic to the $n$-fold connected sum 
\[
(S^3\times S^4)^{\#n}
\]
for every positive integer $n$?

\item Is there a Sasaki-Einstein $7$-manifold diffeomorphic to 
\[
\Sigma\#(S^3\times S^4)^{\#n}
\]
for every positive integer $n$ and every smooth homotopy $7$-sphere $\Sigma$?
\end{enumerate}

On the other hand, \cite[Theorem~4.4]{Fu66} and \cite[Theorem~1.3]{BG67} provide the following topological obstruction to the existence of Sasakian structures.

\begin{theorem}\label{theorem:ob}
Let $M$ be a compact manifold of dimension $2n+1$ admitting a Sasakian structure. Then its $p$-th Betti number is even for odd $p$ with $1\leq p\leq n$, and for even $p$ with $n < p\leq 2n$.
\end{theorem}

It follows that the $2$-connected $7$-manifold
\[
(S^3\times S^4)^{\#(2m+1)}
\]
cannot admit a Sasakian structure for any nonnegative integer $m$.

In this article, we prove the following higher-dimensional analogue of Theorem~\ref{theorem:Kollar}.

\begin{theoremMAIN*}
For every positive integer $m$, there exists a Sasaki--Einstein $7$-manifold homeomorphic to the $2m$-fold connected sum 
\[
(S^3\times S^4)^{\# 2m}.
\]
\end{theoremMAIN*}
Together with the obstruction in Theorem~\ref{theorem:ob}, this completely answers our first question above.

 \section{Links of hypersurface singularities}\label{section:orbifold}

In this section we introduce several isolated $4$-dimensional hypersurface singularities in $\mathbb{C}^5$ whose links play pivotal roles in the present article.

Let $w_0,\ldots,w_n$  be positive integers with no common factor and consider 
a quasi-homogeneous and quasi-smooth polynomial
$
f(x_0,x_1,\ldots,x_n)
$
in the variables $x_0,\ldots,x_n$, where $\mathrm{wt}(x_i)=w_i$.
Let $X$ be the hypersurface defined by $f$ in the weighted projective space
$
\mathbb{P}(\mathbf{w})=\mathbb{P}(w_0,w_1,\ldots,w_n)$.
 The equation
$
f(x_0,x_1,\ldots,x_n)=0
$
also defines a hypersurface $\widehat{X}$ in $\mathbb{C}^{n+1}$, which is smooth away from the origin. The link of $f$ is defined as the intersection
\begin{equation}\label{link}
L(f)=S^{2n+1}_\mathbf{w}\cap \widehat{X},
\end{equation}
where $S^{2n+1}_\mathbf{w}$ is the unit sphere centred at the origin in $\mathbb{C}^{n+1}$.
We also put
$$L^c(f):=S^{2n+1}_\mathbf{w}\setminus L(f).$$
 The map $\Phi:=\frac{f}{|f|}$ on $L^c(f)$ defines a fibration over the unit circle $S^1$. Each fiber is a parallelizable smooth manifold of dimension $2n$ (\cite[Theorem~5.1]{M68}). The boundary of each fiber is exactly the link $L(f)$  (\cite[Lemma 6.1]{M68}).
Therefore, the link is a smooth $(n-2)$-connected $(2n-1)$-manifold that bounds a parallelizable $2n$-manifold (\cite[Theorem 5.2]{M68}).

The key players in this article are the isolated hypersurface singularities of the following types. 

Brieskorn-Pham type:
\begin{equation}\label{BP-type}
\aligned
&f_{m}=x_0^{6m+1}+x_1^{6m+1}+x_2^{6}+x_3^{3}+x_4^{2}. \\
\endaligned
\end{equation}

 Thom-Sebastiani type:
 \begin{equation}\label{TS-type}
\aligned
&g_m=x_0^{6m+4}+x_1^{6m+4}+x_4x_2^{3}+x_3^3+x_1x_4^2,\\
&g_{1,m}=x_0^{6m+3}+x_1^{6m+3}+x_3x_2^4+x_0x_3^3+x_4^2,\\
&g_{2,m}=x_0^{6m+5}+x_1^{6m+5}+x_3x_2^4+x_0x_3^3+x_4^2.\\
\endaligned
\end{equation}

  \section{Topology of links of hypersurface singularities}\label{section:torsionfree}

Following \cite{MO70} and \cite{O71}, this section reviews several results on the topology of links of quasi-homogeneous hypersurface singularities required in what follows.

Retaining the notation of the previous section, let $f$ have degree $d$.  
Alexander  and Poincar\'e dualities yield the  isomorphisms:
\[
    \mathrm{H}_n( L^c(f),\mathbb{Z})\cong \mathrm{H}^n(L(f), \mathbb{Z})\cong \mathrm{H}_{n-1}(L(f),\mathbb{Z}).
\]

Each fiber of the fibration map  $\Phi:L^c(f)\to S^1$  is homotopy equivalent to a wedge of $\mu$   $n$-spheres, where $\mu$ is the Milnor number of $f$.
This fibration can also be viewed as the mapping torus of the monodromy map $$h:\Phi^{-1}(1)\rightarrow\Phi^{-1}(1)$$ given by
\[
    h(x_0,\cdots,x_n)=(\mu_d^{w_0} x_0,\cdots,\mu_d^{w_n}x_n),
\]
where $\mu_d$ is a $d$-th root of unity. Denote by $I$ the identity map on $\Phi^{-1}(1)$. Note that the $d$-fold composition of $h$ is the identity map $I$. The exact sequence for the mapping torus gives
\[
    0\rightarrow \mathrm{H}_{n+1}( L^c(f),\mathbb{Z})\rightarrow \mathrm{H}_n(\Phi^{-1}(1),\mathbb{Z})\xrightarrow{I_\ast-h_\ast} \mathrm{H}_n(\Phi^{-1}(1),\mathbb{Z})
    \rightarrow \mathrm{H}_n(L^c(f),\mathbb{Z})\rightarrow 0.
\]
To determine $\mathrm{H}_n(L^c(f), \mathbb{Z})$, Milnor and Orlik considered the Smith normal form of the matrix $xI_\ast-h_\ast$ over $\mathbb{Q}[x]$. Let $m_1(x),\cdots,m_\mu(x)$ be the invariant factors of $xI_\ast-h_\ast$ over $\mathbb{Q}[x]$. That is, each $m_i(x)$ divides $m_{i+1}(x)$ for $1\leq i<\mu$, and there are invertible matrices $U(x)$ and $V(x)$ over $\mathbb{Q}[x]$ satisfying
\[
    xI_\ast-h_\ast=U(x)\begin{bmatrix}
        m_1(x) & 0 & \cdots & 0\\
        0 & m_2(x) & \cdots & 0 \\
        \vdots & \vdots & \ddots &\vdots \\
        0 & 0 & \cdots & m_\mu(x)
    \end{bmatrix}V(x).
\]
Since $h_\ast^d=I_\ast$, the minimal polynomial $m_\mu(x)$ is the square-free part of the characteristic polynomial $\Delta(x)$ of $h_\ast$. In fact, if the characteristic  polynomial $\Delta(x)$ is of the form $$p_1(x)^{n_1}\cdots p_k(x)^{n_k},$$ where the $p_j(x)$ are irreducible polynomials over $\mathbb{Q}$, then the invariant factors are given explicitly as follows:
\[
    m_i(x)=\prod_{n_j\geq\mu-i+1}p_j(x).
\]
In particular, the corank of $I_\ast-h_\ast$ is the number of zero diagonal entries, which is the multiplicity of $x-1$ in $\Delta(x)$.

Thus, the problem reduces to computing the characteristic polynomial $\Delta(x)$. 
For a monic polynomial $a(x)=\prod_{j=1}^m(x-\alpha_j)\in\mathbb{C}[x]$ with $\alpha_j\in\mathbb{C}^\ast$, its divisor is defined by
\[
    \mathrm{div}(a(x)):=\sum_{j=1}^m\langle\alpha_j\rangle\in\mathbb{Z}[\mathbb{C}^\ast]
\]
in the integral group ring of $\mathbb{C}^*$, where $\langle\alpha_j\rangle$ is the element corresponding to the nonzero complex number $\alpha_j$ in the group ring.
In particular, define $\Lambda_r$ by
\[
    \Lambda_r:=\mathrm{div}(x^r-1)=\langle1\rangle+\langle\mu_r\rangle+\cdots+\langle\mu_r^{r-1}\rangle,
\]
where $\mu_r$ is a primitive $r$-th root of unity.
Note that $$\left(\frac{1}{r_1}\Lambda_{r_1}\right)\cdot\left(\frac{1}{r_2}\Lambda_{r_2}\right)=\frac{1}{\mathrm{lcm}(r_1,r_2)}\Lambda_{\mathrm{lcm}(r_1,r_2)}$$  for any positive integers $r_1$ and $r_2$.

For each $i$, put $$u_i:=\frac{d}{\gcd(d, w_i)}.$$ Then, Milnor and Orlik proved the following:
\begin{theorem}[{\cite[Corollary]{MO70}}]\label{theorem:MO}
    The divisor of the characteristic polynomial $\Delta(x)$ is given by
    \begin{align*}
        \mathrm{div}(\Delta(x))
        &=\left(\frac{d}{u_0w_0}\Lambda_{u_0}-\Lambda_1\right)\left(\frac{d}{u_1w_1}\Lambda_{u_1}-\Lambda_1\right)\cdots\left(\frac{d}{u_nw_n}\Lambda_{u_n}-\Lambda_1\right)\\
        &=\sum(-1)^{n+1-s}\frac{d^s}{w_{i_1}\cdots w_{i_s}\mathrm{lcm} \{u_{i_1},\cdots, u_{i_s}\}}\Lambda_{\mathrm{lcm}\{u_{i_1},\cdots,u_{i_s}\}},
    \end{align*}
    where the sum is taken over all the $2^{n+1}$ subsets $\{i_1,\cdots, i_s\}$ of $\{0,1,\cdots,n\}$.

\end{theorem}

Since $x-1$ occurs in each $x^k-1$ with multiplicity one, one can conclude that
\begin{equation}\label{eq:betti}
    b_{n-1}(L(f))=\sum(-1)^{{n+1}-s}\frac{d^s}{w_{i_1}\cdots w_{i_s}\mathrm{lcm} \{u_{i_1},\cdots, u_{i_s}\}}.
\end{equation}

However, the above argument does not apply to the computation of the torsion subgroup, since $\mathbb{Z}[x]$ is not a principal ideal domain. Orlik conjectured that the Smith normal form of $xI_\ast-h_\ast$ is  given over
$\mathbb{Z}[x]$ (\cite[Conjecture~3.1]{O71}).
This conjecture directly gives the following:
\[
    \mathrm{H}_{n-1}(L(f),\mathbb{Z})=\mathbb{Z}/m_1(1)\mathbb{Z}\oplus\cdots\oplus\mathbb{Z}/m_\mu(1)\mathbb{Z}.
\]

Suppose that Orlik's conjecture holds, i.e., the Smith normal form of  $xI_\ast-h_\ast$ is given in~$\mathbb{Z}[x]$.

For each positive integer $m$, denote by $\Phi_m(x)$ the $m$-th cyclotomic polynomial, i.e.,
\[\Phi_m(x):=\prod_{\substack{1\leq k\leq m \\ \gcd (k,m)=1}}\left(x-e^{\frac{2\pi i k}{m}}\right).\]
It is an irreducible polynomial in $\mathbb{Z}[x]$.
Since the polynomial $x^r-1$ is divisible by $\Phi_m(x)$ in $\mathbb{Z}[x]$ if and only if $r$ is divisible by $m$, it follows from Theorem~\ref{theorem:MO} that
the multiplicity of $\Phi_m(x)$ in $\Delta(x)$ is given by
\[\mathrm{ord}_{\Phi_m(x)}\left(\Delta(x)\right)=\sum_{m|\mathrm{lcm} \{u_{i_1},\cdots, u_{i_s}\}} (-1)^{n+1-s}\frac{d^s}{w_{i_1}\cdots w_{i_s}\mathrm{lcm} \{u_{i_1},\cdots, u_{i_s}\}},\]
 where the sum is taken over all the $2^{n+1}$ subsets $\{i_1,\cdots, i_s\}$ of $\{0,1,\cdots,n\}$. In particular,  if $m$ does not divide 
 $\mathrm{lcm} \{u_0,\cdots, u_n\}$, then $\mathrm{ord}_{\Phi_m(x)}\left(\Delta(x)\right)=0$.
 
 Moreover,
\[m_{\mu-j+1}(x)=\prod_{\mathrm{ord}_{\Phi_m(x)}\left(\Delta(x)\right)\geq j}\Phi_{m}(x)\]
for $1\leq j\leq \mu$. 
Thus, we can compute the integral homology group $\mathrm{H}_{n-1}(L(f),\mathbb{Z})$ explicitly whenever the conjecture holds. 
In particular,
  $\mathrm{H}_{n-1}(L(f),\mathbb{Z})$ is torsion free  if
\begin{equation}\label{eq:torsion2}
\begin{split}
\mathrm{ord}_{\Phi_m(x)}\left(\Delta(x)\right)=\sum_{m|\mathrm{lcm} \{u_{i_1},\cdots, u_{i_s}\}} &(-1)^{n+1-s}\frac{d^s}{w_{i_1}\cdots w_{i_s}\mathrm{lcm} \{u_{i_1},\cdots, u_{i_s}\}}\\
&\leq  b_{n-1}(L(f))=\mathrm{ord}_{\Phi_1(x)}\left(\Delta(x)\right)
\end{split}
\end{equation}
for each positive integer $m$ dividing 
 $\mathrm{lcm} \{u_0,\cdots, u_n\}$ since $\Phi_{1}(x)=x-1$.

In 2022, 
Hertling and Mase verified Orlik's conjecture  for the following classes of singularities (\cite[Theorem~1.3]{HM22}):
\begin{itemize}
    \item Chain type: a quasi-homogeneous polynomial of the form
    \[
        x_0^{a_0+1}+\sum_{i=1}^nx_{i-1}x_i^{a_i};
    \]
    \item Cycle type: a quasi-homogeneous polynomial of the form
    \[
        \sum_{i=1}^nx_{i-1}x_i^{a_i}+x_nx_0^{a_0};
    \]
    \item Thom-Sebastiani sums: if quasi-homogeneous polynomials $f$ and $g$ satisfy Orlik's conjecture, then so does their Thom-Sebastiani sum
    \[ 
        f(x_0,x_1,\cdots,x_{n_f})+g(x_{n_f+1},\cdots,x_{n_f+n_g});
    \]
\end{itemize}
where the $a_i$ are positive integers. Note that Brieskorn-Pham type singularities are covered by the three types above (\cite[Remark~13.1]{HM22}).

Combining these results, we obtain the following proposition.

\begin{proposition}\label{proposition:topology}The third integral homology groups of the links of the hypersurface singularities in \eqref{BP-type} and \eqref{TS-type} are all torsion free. Furthermore, their third Betti numbers are given as follows: 
\[\begin{array}{ll}
b_3(L(f_{m}))=12m,& b_3(L(g_m))=12m+6,\\
b_3(L(g_{1,m}))=6m+2,&
b_3(L(g_{2,m}))=6m+4.\\
\end{array}\]
\end{proposition}
\begin{proof}
It follows from \cite[Theorem~1.3]{HM22} that the Orlik conjecture holds for the polynomials in \eqref{BP-type} and \eqref{TS-type}. Then the statement follows by straightforward computations using \eqref{eq:betti}, 
and ~\eqref{eq:torsion2}.
\end{proof}

\section{Sasaki-Einstein metrics via the $\delta$-invariant}

In \eqref{link}, the unit sphere $S^{2n+1}_{\mathbf{w}}$ carries the Sasakian structure 
induced by the weight~$\mathbf{w} = (a_0, a_1, \dots, a_n)$ (see \cite[\S~1]{BG01} \cite[Example]{Takahashi}).
The configuration is summarized by the following commutative diagram (see~\cite{BN10}):$$
\xymatrix{
L(f)\ar@{->}[d]\ar@{^{(}->}[rr]&&S^{2n+1}_{\mathbf{w}}\ar@{->}[d]\\%
X\ar@{^{(}->}[rr]&&\mathbb{P}(\mathbf{w})}
$$
where the horizontal arrows are Sasakian and Kählerian embeddings, respectively, while the vertical arrows are $S^1$-orbibundles and orbifold Riemannian submersions, respectively.

The following result reduces the problem of finding Sasaki-Einstein metrics on links to the problem of finding K\"ahler-Einstein metrics on Fano orbifolds.
\begin{theorem}[{\cite{BG00}, \cite{BGN03c}, \cite{BGK05}, \cite[Theorem~5]{Kob63}}]\label{theorem:lifting}
If $X$ is a Kähler-Einstein orbifold, then the link $L(f)$ admits a Sasaki-Einstein metric.
\end{theorem}

Therefore, to establish the existence of Sasaki-Einstein metrics on the links of the isolated hypersurface singularities in \eqref{BP-type} and \eqref{TS-type}, it suffices to prove that the following Fano hypersurfaces admit K\"ahler-Einstein metrics:
\begin{equation}\label{hypersurfaces}
\aligned
&\{x_0^{6m+1}+x_1^{6m+1}+x_2^{6}+x_3^{3}+x_4^{2}=0\}&\subset \, &\mathbb{P}(6,6,6m+1,12m+2,18m+3),\\
&\{x_0^{6m+4}+x_1^{6m+4}+x_4x_2^{3}+x_3^3+x_1x_4^2=0\}&\subset \,&\mathbb{P}(6,6,6m+5,12m+8,18m+9),\\
&\{x_0^{6m+3}+x_1^{6m+3}+x_3x_2^4+x_0x_3^3+x_4^2=0\}&\subset \,&\mathbb{P}(12,12,12m+7,24m+8,36m+18),\\
&\{x_0^{6m+5}+x_1^{6m+5}+x_3x_2^4+x_0x_3^3+x_4^2=0\}&\subset \,&\mathbb{P}(12,12,12m+11,24m+16,36m+30). \\
\endaligned
\end{equation}

We now recast the problem of the existence of Kähler–Einstein metrics as a purely algebro-geometric problem.
The fundamental link between K-polystability and the existence of Kähler–Einstein metrics has now been completely established through the works of  \cite{B16}, \cite{BBJ}, \cite{CDS1, CDS2, CDS3},  \cite{Li19}, \cite{LTW19},  \cite{LXZ22},  \cite{Tian15}, \cite{Xu21}, etc..  In particular, we have the following  theorem.
\begin{theorem}\label{theorem}
A Fano orbifold $X$ admits a Kähler–Einstein metric if and only if it is K-polystable.
\end{theorem}

Let $W$ be an $n$-dimensional  projective $\mathbb{Q}$-factorial normal variety and $\Delta_W$ be a $\mathbb{Q}$-divisor on $W$ such that the log pair $(W, \Delta_W)$ has at worst Kawamata log terminal singularities. 
We suppose that $(W,  \Delta_W)$ is a log  $\mathbb{Q}$-Fano variety, i.e., the divisor $-(K_W+ \Delta_W)$ is ample.  Put $L:=-(K_W+ \Delta_W)$.
For a prime divisor $E$ over $W$, the log discrepancy of $E$ with respect to $(W, \Delta_W)$ will be denoted by $A_{(W,\Delta_W)}(E)$. We let
\[S_L(E) =\frac{1}{L^n} \int_0^{\infty} \mathrm{vol}_{\widetilde{W}}(\varphi^*(L) - u E)\, du,\]
where $\varphi:\widetilde{W}\to W$ is a birational morphism such that $E$ is a divisor on $\widetilde{W}$.

We now let $\mathfrak{G}$ be an algebraic subgroup of $\mathrm{Aut}(W,\Delta_W)$ and let $Z$ be a closed $\mathfrak{G}$-invariant subvariety of $W$. We first define the local $\mathfrak{G}$-equivariant $\delta$-invariant along $Z$ as follows:
\[\delta_{\mathfrak{G}, Z}(W,\Delta_W)=\inf_{E/W}\frac{A_{(W,\Delta_W)}(E)}{S_{L}(E)},\]
where the infimum is taken  over all the $\mathfrak{G}$-invariant prime divisors $E$ over $W$ whose centers contain the subvariety $Z$.
When $\mathfrak{G}$ is trivial, we write $\delta_Z(W,\Delta_W)$ instead of $\delta_{\mathfrak{G},Z}(W,\Delta_W)$ and call it the local $\delta$-invariant along $Z$. When $Z=\emptyset$, we write $\delta_{\mathfrak{G}}(W,\Delta_W)$ instead of $\delta_{\mathfrak{G},\emptyset}(W,\Delta_W)$ and call it the $\mathfrak{G}$-equivariant $\delta$-invariant. Finally, when both $\mathfrak{G}$ is trivial and $Z=\emptyset$, we write $\delta(W,\Delta_W)$ instead of $\delta_{\mathfrak{G},\emptyset}(W,\Delta_W)$ and call it the $\delta$-invariant of $(W,\Delta_W)$.

Using the $\delta$-invariant, Blum–Jonsson \cite{Blum-Jonsson17} and Fujita–Odaka \cite{Fujita-Odaka16} developed an algebro-geometric criterion for K-stability. In particular, in view of the result of \cite{Z21}, the following criterion will be one of the key ingredients in establishing K-polystability of the Fano hypersurfaces considered in \eqref{hypersurfaces}.

\begin{theorem}[{\cite[Corollary~4.14]{Z21}}]\label{theorem:G-delta}
Let $(W, \Delta_W)$ be a  log $\mathbb{Q}$-Fano variety with a reductive algebraic subgroup~$\mathfrak{G}$ of $\mathrm{Aut}(W, \Delta_W)$.
A log $\mathbb{Q}$-Fano variety $(W, \Delta_W)$ is K-polystable  if $\delta_\mathfrak{G}(W, \Delta_W)>1$.
\end{theorem}


 \section{K-stability of weighted Fano hypersurface 3-folds}\label{section:K-polystability}

In this section, we verify the K-polystability of Fano hypersurfaces defined by quasihomogeneous polynomials of the form \eqref{TS-type} in the $4$-dimensional weighted projective space $\mathbb{P}(a_0,a_1,a_2,a_3,a_4)$.
These hypersurfaces admit cyclic covering structures, which allow us to reduce the problem to the study of log Fano pairs on their base spaces (\cite[Theorem~1.2]{LZ22}).
To prove K-polystability, we apply the Abban--Zhuang method, systematically reformulated by Fujita (see \cite{book,Fu23}), to the log Fano pairs arising from the cyclic covering structures of the hypersurfaces. Throughout the application of the Abban--Zhuang method, we adopt the notation of \cite{Fu23}. In particular, we refer to \cite[Theorems~4.8 and~4.17]{Fu23} for the notation used in \eqref{eq:2dot}, \eqref{eq:3dot}, \eqref{eq:W2dot}, \eqref{eq:W3dot}, \eqref{eq:WS3dot}, \eqref{eq:Extra-2dot}, \eqref{eq:Extra-3dot}, \eqref{eq:G2-2dot}, \eqref{eq:G2-3dot}, \eqref{eq:Extra-G2-2dot}, and \eqref{eq:Extra-G2-3dot}.
To estimate  local $\delta$-invariants, we apply
 \cite[Corollary~4.8]{Fu23}, whose setting can be relaxed by \cite[Theorem~11.14]{Fu23}.

Meanwhile, the K-polystability of Fano hypersurfaces defined by quasihomogeneous polynomials of the form \eqref{BP-type} is completely determined by \cite[Theorem~1.4]{LST25}.

We use coordinates $x_i$ on $\mathbb{P}(a_0,a_1,a_2,a_3,a_4)$, where $\mathrm{wt}(x_i)=a_i$. When considering weighted projective spaces obtained from the original space $\mathbb{P}(a_0,a_1,a_2,a_3,a_4)$, we retain the notation $x_i$ for the coordinate induced by $x_i$, with $\mathrm{wt}(x_i)=a_i$, unless this causes confusion.

\begin{theorem} \label{theorem:g}
For each integer $n\geq 0$
the hypersurface $X$ of degree $36n+24$  defined by 
\[x_0^{6n+4}+x_1^{6n+4}+x_2^{3}x_4+x_3^3+x_4^2x_1=0\]
in $\mathbb{P}(6,6,6n+5,12n+8,18n+9)$
is K-polystable.
\end{theorem}
\begin{proof}
The hypersurface  $X$ is a cyclic triple  cover of $\mathbb{P}(6,6,6n+5,18n+9)$ branched along the surface defined by
\[x_0^{6n+4}+x_1^{6n+4}+x_2^{3}x_4+x_4^2x_1=0.\]
The weighted projective space $\mathbb{P}(6,6,6n+5,18n+9)$ together with the branch surface is isomorphic to $\mathbb{P}:=\mathbb{P}(2,2,6n+5,6n+3)$ with the surface $B$
defined by
\[x_0^{6n+4}+x_1^{6n+4}+x_2x_4+x_4^2x_1=0.\]
Therefore, the K-polystability of $X$ is equivalent to the K-polystability of the log Fano pair
$$\left(\mathbb{P}, \frac{2}{3}B+\frac{2}{3}H\right),$$ where $H$ is the hypersurface  of $\mathbb{P}$ defined by $x_2=0$ (\cite[Theorem~1.2]{LZ22}).

Denote by $A$ the divisor  class given by $\mathcal{O}_\mathbb{P}(1)$.
Set $\Delta_\mathbb{P} :=\frac{2}{3}B+\frac{2}{3}H$ and $L=-\left(K_\mathbb{P}+\Delta_\mathbb{P}\right)$.  We have $L\sim_\mathbb{Q} \frac{10}{3}A$.

Let $\Pi$ be the hypersurface of $\mathbb{P}$ defined by $x_0=0$. It is isomorphic to the weighted projective plane 
$\mathbb{P}(2,6n+5,6n+3)$. We identify $\Pi$ with this weighted projective plane.
The surface $\Pi$ contains three singular points at 
\[
o_1=[0:1:0:0], \qquad
o_2=[0:0:1:0], \qquad
o_4=[0:0:0:1]
\]
in $\mathbb{P}$.
The first singularity $o_1$ is of type $\frac{1}{2}(1,1)$, $o_2$ is of type
$\frac{1}{6n+5}(2,6n+3)$, and $o_4$ is of type $\frac{1}{6n+3}(1,1)$.
 Denote the curve $x_2=0$ on $\Pi$ by $\ell_H$, and denote the curve defined
 by 
 \[x_1^{6n+4}+x_2x_4+x_4^2x_1=0\]
 on $\Pi$ by $\ell_B$.
 The curves $\ell_H$ and $\ell_B$ intersect only at  a smooth point $o_B$ with intersection number $1$ and at $o_4$ with intersection number $\frac{1}{6n+3}$. We set $A_\Pi=A|_\Pi$.
 
 We first  obtain
 \[\left(K_\mathbb{P}+\frac{2}{3}B+\frac{2}{3}H+\Pi\right)\Big|_\Pi=K_\Pi+\frac{2}{3}\ell_B+\frac{2}{3}\ell_H.\]
Set $\Delta_\Pi=\frac{2}{3}\ell_B+\frac{2}{3}\ell_H$.
Since the pair $(\Pi, \Delta_\Pi)$ is Kawamata log terminal, the pair $(\mathbb{P},\Delta_\Pi+\Pi)$ is purely log terminal along $\Pi$.
We have
\[
S_L(\Pi) =\frac{1}{L^3} \int_0^{\infty} \mathrm{vol}_\mathbb{P}(L - u \Pi) \,du=
\int_0^{\frac{5}{3}}  \left(1 - \frac{3}{5} u\right)^3 \,du=\frac{5}{12}.
\]

Let $p$ be a point on $\Pi$.  
For an irreducible curve $\mathfrak{c}$ in the linear system $|mA_\Pi|$ passing through $p$,  we have
\begin{equation}\label{eq:2dot}
\begin{split}
S_L  (V_{\bullet, \bullet}^{\Pi}; \mathfrak{c}) 
&= \frac{3}{L^3} \int_0^{\frac{5}{3}} \int_0^{\infty} \mathrm{vol}_{\Pi} \left(\left(L - u \Pi\right)|_{\Pi} - v \mathfrak{c}\right)\,dv \,du \\
&= \frac{3}{L^3} \int_0^{\frac{5}{3}} \int_0^{\frac{10-6u}{3m}} \left(\frac{10}{3}-2u- mv \right)^2 A_\Pi^2 \,dv \,du \\
&= \frac{5}{6m},
\end{split}
\end{equation}
and
\begin{equation}\label{eq:3dot}
\begin{split}
S_L (W_{\bullet, \bullet, \bullet}^{\Pi, \mathfrak{c}}; p) 
&= \frac{3}{L^3}\int_0^{\frac{5}{3}}  \int_0^{\frac{10-6u}{3m}}\left( \left(\left(L - u \Pi\right)|_{\Pi} - v \mathfrak{c}\right) \cdot\mathfrak{c}\right)^2\, dv \, du \\
&= A_\Pi^2m^2S_L  (V_{\bullet, \bullet}^{\Pi}; \mathfrak{c})  \\
&= \frac{5m}{12(6n+3)(6n+5)}.
\end{split}
\end{equation}

Let $\ell$ be a curve in the pencil $|(6n+5)A_\Pi|$ passing through the point $p$.

First, suppose that the curve $\ell$ is given by $x_2=\alpha x_1x_4$ for some nonzero constant $\alpha$.\
We then have 
\[\left(K_\Pi+\Delta_\Pi+\ell\right)\big|_\ell= 
\begin{cases}
K_\ell+\frac{5}{6}o_1+ \frac{18n+10}{18n+9}o_4 +\frac{2}{3}p_B & \text{if $\alpha\ne -1$}, \\
K_\ell+\frac{5}{6}o_1+ \frac{30n+16}{18n+9}o_4 & \text{if $\alpha= -1$},
\end{cases}
\]
where $p_B$ denotes the smooth intersection point of $\ell$ and $\ell_B$ when $\alpha\ne -1$.
Therefore, the pair $(\Pi, \Delta_\Pi+\ell)$ is purely log terminal in a neighborhood of  the point $p$, provided $p\ne o_4$. 
By \cite[Corollary~4.8]{Fu23},  \eqref{eq:2dot} and \eqref{eq:3dot} imply that
\[\delta_p\left(\mathbb{P}, \Delta_\mathbb{P}\right)\geq \min\left\{\frac{12}{5}, \frac{6(6n+5)}{5}, \frac{2(6n+3)}{5}\right\}\geq\frac{6}{5}\]
for every point $p\in \ell\setminus\{o_4\}$.

Next, suppose that $\alpha=0$, so that $\ell=\ell_H$. In this case, $A_{(\Pi, \Delta_\Pi)}(\ell_H)=\frac{1}{3}$, and
we obtain
\[\left(K_\Pi+\Delta_\Pi+\left(A_{(\mathbb{P}, \Delta_\mathbb{P})}(\ell_H)\right)\ell_H\right)\big|_{\ell_H}= K_\ell+\frac{1}{2}o_1+ \frac{18n+8}{18n+9}o_4+\frac{2}{3}o_B.
\]
The pair $(\Pi, \frac{2}{3}\ell_B+\ell_H)$ is purely log terminal.
Since $\ell_H \in |(6n+5)A_\Pi|$, it follows from \eqref{eq:2dot}, \eqref{eq:3dot}, and \cite[Corollary~4.8]{Fu23} that
 \[\delta_p\left(\mathbb{P}, \Delta_\mathbb{P}\right)\geq \min\left\{\frac{12}{5}, \frac{2(6n+5)}{5}, \frac{4(6n+3)}{5}\right\}\geq\frac{6}{5}\]
for every point $p\in \ell_H\setminus\{o_4\}$.

Now let $\ell_1$ be the curve on $\Pi$ defined by $x_1=0$. We have
\[\left(K_\Pi+\Delta_\Pi+\ell_1\right)\big|_{\ell_1}= K_{\ell_1}+\frac{18n+14}{18n+15}o_2+ \frac{18n+10}{18n+9}o_4. 
\]
Thus, the pair $(\Pi, \Delta_\Pi+\ell_1)$ is purely log terminal around every point
$p\in\ell_1\setminus\{o_4\}$. 
Since $\ell_1 \in |2A_\Pi|$, we obtain
\[\delta_p\left(\mathbb{P}, \Delta_\mathbb{P}\right)\geq \min\left\{\frac{12}{5}, \frac{2(6n+5)}{5}, \frac{2(6n+3)}{5}\right\}\geq\frac{6}{5}\]
for every point $p\in \ell_1\setminus\{o_4\}$.

Finally, let $\ell_4$ be the curve on $\Pi$ defined by $x_4=0$.
We have
\[\left(K_\Pi+\Delta_\Pi+\ell_4\right)\big|_{\ell_4}= K_{\ell_4}+\frac{5}{6}o_1+ \frac{5(6n+4)}{3(6n+5)}o_2.
\]
The pair $(\Pi, \Delta_\Pi+\ell_4)$ is purely log terminal around $p$, unless $p=o_2$.
Since $\ell_4 \in |(6n+3)A_\Pi|$, we obtain
\[\delta_p\left(\mathbb{P}, \Delta_\mathbb{P}\right)\geq \min\left\{\frac{12}{5}, \frac{6(6n+3)}{5}, \frac{2(6n+5)}{5}\right\}\geq\frac{6}{5}\]
for every point $p\in \ell_4\setminus\{o_2\}$.

Thus far, we have verified that $\delta_p\left(\mathbb{P}, \Delta_\mathbb{P}\right)>\frac{6}{5}$ for every point $p\in \Pi\setminus\{o_4\}$.

To treat the remaining point $o_4$, consider the blow up $\varphi:\widetilde{\Pi}\to \Pi$ at $o_4$, and denote its exceptional curve  by $\mathfrak{e}$. Let  $\widetilde{\ell}_1$, $\widetilde{\ell}_B$, $\widetilde{\ell}_H$ be the proper transforms of $\ell_1$, $\ell_B$, $\ell_H$ by $\varphi$, respectively. Also let $q_1$, $q_B$, $q_H$ denote the intersection points of $\widetilde{\ell}_1$, $\widetilde{\ell}_B$, $\widetilde{\ell}_H$ with $\mathfrak{e}$, respectively.

We have $A_{(\Pi,\Delta_\Pi)}(\mathfrak{e})=\frac{2}{18n+9}$. Furthermore,   the adjunction formula yields
\[\left(K_{\widetilde{\Pi}}+\frac{2}{3}\widetilde{\ell}_B+\frac{2}{3}\widetilde{\ell}_H+\mathfrak{e}\right)\Big|_\mathfrak{e}=K_\mathfrak{e}+\frac{2}{3}q_B+\frac{2}{3}q_H.\]

For $0\leq u\leq \frac{5}{3}$,  the divisor $\varphi^*\left(\left(\frac{10}{3}-2u\right) A_\Pi \right)- v \mathfrak{e}$ is pseudoeffective if and only if  $$v \leq \tau(u):=\frac{5-3u}{9(2n+1)}.$$ Its Zariski decomposition is given by
\[
\begin{split}
  P (u, v) &=
\begin{dcases}
\varphi^*\left(\left(\frac{10}{3}-2u\right) A_\Pi \right)- v \mathfrak{e}  &  \mbox{on } 0 \leq v \leq  \frac{10-6u}{9(2n+1)(6n+5)}  \\
\varphi^*\left(\left(\frac{10}{3}-2u\right) A_\Pi \right)- v \mathfrak{e}-k_n(u,v)\widetilde{\ell}_1 &
  \mbox{on }  \frac{10-6u}{9(2n+1)(6n+5)}  \leq v \leq \tau(u)\end{dcases} \\
N (u, v) &= 
\begin{dcases}
0 &   \mbox{on } 0 \leq v \leq  \frac{10-6u}{9(2n+1)(6n+5)} \\
k_n(u,v)\widetilde{\ell}_1, &    \mbox{on }  \frac{10-6u}{9(2n+1)(6n+5)}  \leq v \leq \tau(u),\end{dcases}
\end{split}
\]
where $P (u, v)$ is the positive part, $N (u, v) $ is the negative part, and 
$$k_n(u,v)=\frac{9(2n+1)(6n+5)v-(10-6u)}{9(2n+1)}.$$
We then compute
\begin{equation}\label{eq:W2dot}
\begin{split}
S_L (V_{\bullet, \bullet}^{\widetilde{\Pi}}; \mathfrak{e}) 
&= \frac{3}{L^3} \int_0^{\frac{5}{3}} \int_0^{\infty} \mathrm{vol}_{\widetilde{\Pi} }\left(\varphi^*\left(\left(\frac{10}{3}-2u\right) A_\Pi \right)- v \mathfrak{e} \right)\,dv \,du \\
&=  \frac{3}{L^3}\int_0^{\frac{5}{3}}  \int_0^{\tau(u)}P(u,v)^2\,dv \,du \\
&= \frac{5(6n+7)}{36(2n+1)(6n+5)}.
\end{split}
\end{equation}

For every point $q\in \mathfrak{e}\setminus\{q_1\}$, we have
\begin{equation}\label{eq:W3dot}
\begin{split}
S_L (W_{\bullet, \bullet, \bullet}^{\widetilde{\Pi}, \mathfrak{e}}; q) 
&=  \frac{3}{L^3}\int_0^{\frac{5}{3}}  \int_0^{\tau(u)}\left(P(u,v)\cdot\mathfrak{e}\right)^2\,dv 
\,du \\
&= \frac{5}{6(6n+5)}.
\end{split}
\end{equation}
For the point $q_1$, we compute
\begin{equation}\label{eq:WS3dot}
\begin{split}
S_L (W_{\bullet, \bullet, \bullet}^{\widetilde{\Pi}, \mathfrak{e}}; q_1) 
&=  \frac{3}{L^3}\int_0^{\frac{5}{3}}  \int_0^{\tau(u)}\left[\left(P(u,v)\cdot\mathfrak{e}\right)^2+2\left(P(u,v)\cdot\mathfrak{e} \right)\mathrm{ord}_{q_1}\left(N(u,v)\right)\right]
\,dv \,du \\
&= \frac{5}{12}.
\end{split}
\end{equation}
It then follows from \cite[Corollary~4.8]{Fu23} that
\[\delta_{o_4}\left(\mathbb{P}, \Delta_\mathbb{P}\right)\geq \min\left\{\frac{12}{5}, \frac{8(6n+5)}{5(6n+7)}, \frac{2(6n+5)}{5}, \frac{12}{5}\right\}\geq\frac{6}{5}.\]
Consequently,  we have verified that
$\delta_{p}\left(\mathbb{P}, \Delta_\Pi\right)\geq\frac{6}{5}$ for every point $p$ in $\Pi$.

Now we consider the point $o_0:=[1:0:0:0]$ in $\mathbb{P}$. The surface $H$ is isomorphic to $\mathbb{P}(1,1, 6n+3)$, and we  identify $H$ with this weighted projective plane. Let $A_H$ be the divisor class on $H$ given by $\mathcal{O}_H(1)$, and note that $A|_H=\frac{1}{2}A_H$.
By the adjunction formula, 
\[\left(K_\mathbb{P}+\frac{2}{3}B+H\right)\Big|_{H}=K_{H}+\frac{1}{2}\ell_4'+\frac{2}{3}\ell_B'.\]
where $\ell_4'$ is the curve on $H$ defined by $x_4=0$ and  $\ell_B':=B|_{H}$.
The curve $\ell_B'$ is defined by
\[x_0^{6n+4}+x_1^{6n+4}+x_4x_1=0\]
in $H=\mathbb{P}(1,1, 6n+3)$. 
The curves $\ell_4'$ and $\ell_B'$ intersects at $6n+4$ distinct smooth  points, $p_1, \ldots, p_{6n+4}$ of $H$, all different from $o_0$.
Note that the pair $(\mathbb{P}, \frac{2}{3}B+H)$ is purely log terminal, and the log discrepancy of $H$ with respect to $(\mathbb{P},\Delta_\mathbb{P})$ is $\frac{1}{3}$.

Since
\[\left(K_{H}+\frac{2}{3}\ell_B'+\ell_4'\right)\Big|_{\ell_4'}=K_{\ell_4'}+\frac{2}{3}\left(p_1+\cdots+p_{6n+4}\right),\]
we see that the pair $(H, \frac{2}{3}\ell_B'+\ell_4')$ is purely log terminal.
We obtain
\begin{equation}\label{eq:Extra-1dot}
S_L(H) =\frac{1}{L^3} \int_0^{\infty} \mathrm{vol}_\mathbb{P}(L - u H) \,du=
\int_0^{\frac{10}{3(6n+5)}}  \left(1 - \frac{3(6n+5)}{10} u\right)^3 \,du=\frac{5}{6(6n+5)},
\end{equation}
\begin{equation}\label{eq:Extra-2dot}
\begin{split}
S_L  (V_{\bullet, \bullet}^{H}; \mathfrak{\ell_4'}) 
&= \frac{3}{L^3} \int_0^{\frac{10}{3(6n+5)}} \int_0^{\infty} \mathrm{vol}_{H} \left(\left(L - u H\right)|_{H} - v \ell_4'\right) \,dv \,du \\
&= \frac{3}{L^3} \int_0^{\frac{10}{3(6n+5)}} \int_0^{\frac{10-3(6n+5)u}{6(6n+3)}} \mathrm{vol}_{H} \left(\left(\frac{10}{3} - (6n+5)u \right)A|_{H} - (6n+3)v A_H\right) \,dv \,du \\
&= \frac{3}{L^3} \int_0^{\frac{10}{3(6n+5)}} \int_0^{\frac{10-3(6n+5)u}{6(6n+3)}} \left(\frac{5}{3}-\frac{6n+5}{2}u- (6n+3)v \right)^2 A_H^2 \,dv \,du \\
&= \frac{5}{12(6n+3)},
\end{split}
\end{equation}
and
\begin{equation}\label{eq:Extra-3dot}
\begin{split}
S_L (W_{\bullet, \bullet, \bullet}^{H, \ell_4'}; o_0) 
&= \frac{3}{L^3}\int_0^{\frac{10}{3(6n+5)}}  \int_0^{\frac{10-3(6n+5)u}{6(6n+3)}}\left( \left(\left(L - u H\right)|_{H} - v \ell_4'\right) \cdot\ell_4'\right)^2\,dv \,du \\
&= A_H^2(6n+3)^2  S_L  (V_{\bullet, \bullet}^{H}; \mathfrak{\ell_4'}) \\
&= \frac{5}{12}.
\end{split}
\end{equation}

By \cite[Corollary~4.8]{Fu23}, the values \eqref{eq:Extra-1dot}, \eqref{eq:Extra-2dot} and \eqref{eq:Extra-3dot} yield
\[\delta_{o_0}\left(\mathbb{P}, \Delta_\mathbb{P}\right)\geq \min\left\{\frac{2(6n+5)}{5},  \frac{6(6n+3)}{5}, \frac{12}{5}\right\}\geq\frac{6}{5}.\]

The automorphism group of $(\mathbb{P}, \frac{2}{3}B+\frac{2}{3}H)$ contains the subgroup $\mathfrak{G}$ generated by the  automorphisms
\[\aligned
&[x_0:x_1:x_2:x_4]\mapsto [\mu_{6n+4}x_0:x_1: x_2: x_4],\\
&[x_0:x_1:x_2:x_4]\mapsto [x_0:x_1:-x_2: - x_4],\\
\endaligned
\]
where $\mu_{6n+4}$ is a primitive $(6n+4)$-th root of unity.
All  $\mathfrak{G}$-invariant points except the point~$o_0$ lie on the surface $\Pi$.
Moreover, every irreducible $\mathfrak{G}$-invariant subvariety  of  positive dimension has an intersection point with $\Pi$. Therefore, the inequality $\delta_p\left(\mathbb{P}, \Delta_\mathbb{P}\right)\geq\frac{6}{5}$ for every point $p$ in $\Pi\cup\{o_0\}$ implies that
$\delta_{\mathfrak{G}}(\mathbb{P}, \Delta_\mathbb{P})\geq \frac{6}{5}>1$.
By Theorem~\ref{theorem:G-delta}, it follows that
the pair $(\mathbb{P},\Delta_\mathbb{P})$ is K-polystable. \end{proof}

\begin{theorem}\label{theorem:h}
For each integer $n\geq 4$,  the hypersurface $Y$ of degree $12n$  defined by 
\[x_0^{n}+x_1^{n}+x_3x_2^4+x_1x_3^3+x_4^2=0\]
in $\mathbb{P}(12,12,2n+1,4n-4,6n)$
is K-polystable.
\end{theorem}
\begin{proof}
First, suppose that $n\not\equiv 1$ (mod $3$). 

The hypersurface $Y$ is a double cover of $\mathbb{P}(12,12,2n+1,4n-4)$ branched along the surface
\[
x_0^{n}+x_1^{n}+x_3x_2^4+x_1x_3^3=0.
\]
After dividing the three weights by their common factor, the pair consisting of the weighted projective space $\mathbb{P}(12,12,2n+1,4n-4)$ and its branch divisor is isomorphic to the pair $(\mathbb{P},B)$, where
$
\mathbb{P}:=\mathbb{P}(3,3,2n+1,n-1)$
and $B$ is the surface defined by
\[
x_0^{n}+x_1^{n}+x_3x_2+x_1x_3^3=0
\]
in $\mathbb{P}$.
Consequently, the K-polystability of $Y$ is equivalent to that of the log Fano pair
\[
\left(\mathbb{P}, \frac{1}{2}B+\frac{3}{4}H\right),
\]
where $H\subset \mathbb{P}$ is the hypersurface defined by $x_2=0$.

Let $A$ denote the divisor class corresponding to $\mathcal{O}_{\mathbb{P}}(1)$. Define
$\Delta_{\mathbb{P}}:=\frac{1}{2}B+\frac{3}{4}H$ and~$L:=-\left(K_{\mathbb{P}}+\Delta_{\mathbb{P}}\right)$.
Then
$
L\sim_{\mathbb{Q}}\frac{21}{4}A$.

Let $\Pi$  be the hypersurface of $\mathbb{P}$ defined by $x_0=0$. We identify $\Pi$ with the weighted projective plane $\mathbb{P}(3,2n+1,n-1)$ and set
\[
A_{\Pi}:=A|_{\Pi}, \qquad
\ell_B:=B|_{\Pi}, \qquad
\ell_H:=H|_{\Pi}.
\]
The curves $\ell_B$ and $\ell_H$ are defined by
\[
x_1^{n}+x_3x_2+x_1x_3^3=0
\qquad\text{and}\qquad
x_2=0,
\]
respectively. The surface $\Pi$ has three quotient singular points at
\[
o_1=[0:1:0:0], \qquad
o_2=[0:0:1:0], \qquad
o_3=[0:0:0:1]
\]
in $\mathbb{P}$.
By the adjunction formula,
\[
\left(K_{\mathbb{P}}+\frac{1}{2}B+\frac{3}{4}H+\Pi\right)\Big|_{\Pi}
=
K_{\Pi}+\frac{1}{2}\ell_B+\frac{3}{4}\ell_H.
\]
Furthermore,
\begin{equation}\label{eq:G2-LH}
\left(K_{\Pi}+\frac{1}{2}\ell_B+\ell_H\right)\Big|_{\ell_H}
=
K_{\ell_H}
+\frac{2}{3}o_1
+\frac{2n-3}{2n-2}o_3
+\frac{1}{2}o_B,
\end{equation}
where $o_B$ is the unique smooth intersection point of $\ell_B$ and $\ell_H$. It follows that the pair
$
\left(\Pi,\frac{1}{2}\ell_B+\ell_H\right)
$
is purely log terminal along the curve $\ell_H$. Since $\ell_B$ is quasi-smooth, this implies that 
$
\left(\Pi,\frac{1}{2}\ell_B+\frac{3}{4}\ell_H\right)
$
is Kawamata log terminal, and therefore
$
(\mathbb{P},\Delta_{\mathbb{P}}+\Pi)
$
is purely log terminal along $\Pi$.
Moreover, we compute
\[
S_L(\Pi)
=
\frac{1}{L^3}
\int_{0}^{\infty}
\operatorname{vol}_{\mathbb{P}}(L-u\Pi)\,du
=
\int_{0}^{\frac{7}{4}}
\left(1-\frac{4}{7}u\right)^3\,du
=
\frac{7}{16}.
\]

Let  $p$ be a point on $\Pi$, and let $\mathfrak{c}\in |mA_\Pi|$ be an irreducible curve passing through $p$. We have
\begin{equation}\label{eq:G2-2dot}
\begin{split}
S_L(V_{\bullet,\bullet}^{\Pi};\mathfrak{c})
&=
\frac{3}{L^3}
\int_0^{\frac{7}{4}}
\int_0^{\infty}
\operatorname{vol}_{\Pi}\left((L-u\Pi)|_{\Pi}-v\mathfrak{c}\right)\,dv\,du \\
&=
\frac{3}{L^3}
\int_0^{\frac{7}{4}}
\int_0^{\frac{21-12u}{4m}}
\left(\frac{21}{4}-3u-mv\right)^2A_\Pi^2\,dv\,du \\
&=
\frac{21}{16m},
\end{split}
\end{equation}
and
\begin{equation}\label{eq:G2-3dot}
\begin{split}
S_L(W_{\bullet,\bullet,\bullet}^{\Pi,\mathfrak{c}};p)
&=
\frac{3}{L^3}
\int_0^{\frac{7}{4}}
\int_0^{\frac{21-12u}{4m}}
\left(\left((L-u\Pi)|_{\Pi}-v\mathfrak{c}\right)\cdot\mathfrak{c}\right)^2\,dv\,du \\
&=
A_\Pi^2m^2
S_{(\mathbb{P}, \Delta_\mathbb{P})}(V_{\bullet,\bullet}^{\Pi};\mathfrak{c}) \\
&=
\frac{7m}{16(2n+1)(n-1)}.
\end{split}
\end{equation}

Set
$
\Delta_{\Pi}:=\frac{1}{2}\ell_B+\frac{3}{4}\ell_H$.
Let $\ell$ be a curve in the pencil $|(2n+1)A_\Pi|$ passing through $p$.

First,  suppose  that $\ell$ is defined by the equation
$
x_1x_3^2=\alpha x_2
$
for some nonzero constant~$\alpha$. We then have
\[
\left(K_\Pi+\Delta_\Pi+\ell\right)\big|_\ell=
\begin{cases}
K_\ell+\dfrac{7}{6}o_1+\dfrac{4n-3}{4n-4}o_3+\dfrac{1}{2}p_B,
& \text{if }\alpha\neq -1, \\[1ex]
K_\ell+\dfrac{7}{6}o_1+\dfrac{6n-5}{4n-4}o_3,
& \text{if }\alpha=-1,
\end{cases}
\]
where $p_B$ is the unique smooth intersection point of $\ell$ and $\ell_B$ in the case $\alpha\neq -1$. 
Consequently, the pair $(\Pi,\Delta_\Pi+\ell)$ is purely log terminal in a neighborhood of the point $p$ unless $p$ is a singular point.
Applying \cite[Corollary~4.8]{Fu23} together with \eqref{eq:G2-2dot} and \eqref{eq:G2-3dot}, we obtain
\[
\delta_{p}\left(\mathbb{P},\Delta_\mathbb{P}\right)
\ge
\min\left\{
\frac{16}{7},
\frac{16(2n+1)}{21},
\frac{8(n-1)}{7}
\right\}
\ge
\frac{8}{7}
\]
for every point $p\in \ell\setminus\{o_1, o_3\}$.

Next, suppose that $\ell=\ell_H$. The log discrepancy of $\ell_H$ with respect to $(\Pi,\Delta_\Pi)$ is~$\frac{1}{4}$, and for each point $p$ in $\ell_H$ the log 
discrepancy of $p$ with respect to $(\ell_H,\Delta_{\ell_H})$ is at least~$\frac{1}{2(n-1)}$, where $\Delta_{\ell_H}=\frac{2}{3}o_1
+\frac{2n-3}{2n-2}o_3
+\frac{1}{2}o_B$ as in \eqref{eq:G2-LH}.
Then, \cite[Corollary~4.8]{Fu23}, together with \eqref{eq:G2-2dot}, and \eqref{eq:G2-3dot}, immediately yields
\[
\delta_{p}\left(\mathbb{P},\Delta_\mathbb{P}\right)
\ge
\min\left\{
\frac{16}{7},
\frac{4(2n+1)}{21},
\frac{8}{7}
\right\}
\ge
\frac{8}{7}
\]
for every point $p\in\ell_H$. In particular,  $\delta_{o_1}\left(\mathbb{P},\Delta_\mathbb{P}\right)\geq \frac{8}{7}$ and $ \delta_{o_3}\left(\mathbb{P},\Delta_\mathbb{P}\right)\geq \frac{8}{7}$.

Now consider the curve $\ell_3\in |(n-1)A_\Pi|$ defined by $x_3=0$, and suppose that the  point $p$ lies on $\ell_3$. We have
\[
\left(K_\Pi+\Delta_\Pi+\ell_3\right)\big|_{\ell_3}
=
K_{\ell_3}
+\frac{11}{12}o_1
+\frac{5n}{4n+2}o_2.
\]
Again, \eqref{eq:G2-2dot} and \eqref{eq:G2-3dot} imply that
\[
\delta_{p}\left(\mathbb{P},\Delta_\mathbb{P}\right)
\ge
\min\left\{
\frac{16}{7},
\frac{16(n-1)}{21},
\frac{4(2n+1)}{21}
\right\}
\ge
\frac{8}{7}
\]
for every $p\in\ell_3\setminus\{o_2\}$.

Next, consider the curve $\ell_1\in |3A_\Pi|$ defined by $x_1=0$, and suppose that $p\in \ell_1$. We have
\[
\left(K_\Pi+\Delta_\Pi+\ell_1\right)\big|_{\ell_1}
=
K_{\ell_1}
+\frac{4n+1}{4n+2}o_2
+\frac{4n-3}{4n-4}o_3.
\]
It follows from \eqref{eq:G2-2dot} and \eqref{eq:G2-3dot} that
\[
\delta_{p}\left(\mathbb{P},\Delta_\mathbb{P}\right)
\ge
\min\left\{
\frac{16}{7},
\frac{8(n-1)}{21}
\right\}
\ge
\frac{8}{7}
\]
for every point $p\in \ell_1\setminus\{o_3\}$. 

Consequently,  we have verified that
$\delta_{p}\left(\mathbb{P}, \Delta_\Pi\right)\geq\frac{8}{7}$ for every point $p$ in $\Pi$.

Now consider the point $o_0:=[1:0:0:0]$ in $\mathbb{P}$. The surface $H$ is isomorphic to $\mathbb{P}(1,1,n-1)$, and we identify $H$ with $\mathbb{P}(1,1,n-1)$. Let $A_H$ denote the divisor class on~$H$ corresponding to $\mathcal{O}_H(1)$. Note that
\[
A|_H=\frac{1}{3}A_H.
\]
We have
\[
\left(K_\mathbb{P}+\frac{1}{2}B+H\right)\Big|_{H}
=K_H+\frac{2}{3}\ell_3'+\frac{1}{2}\ell_B',
\]
where $\ell_3'$ is the curve on $H$ defined by $x_3=0$, and $\ell_B':=B|_H$. The curve $\ell_B'$ is defined by
$
x_0^n+x_1^n+x_1x_3=0
$
in $H=\mathbb{P}(1,1,n-1)$.

The curves $\ell_3'$ and $\ell_B'$ intersect at $n$ distinct smooth points
$
p_1,\ldots,p_n
$
of $H$, all of which are different from $o_0$. Note that the pair $(\mathbb{P},\frac{1}{2}B+H)$ is purely log terminal, and the log discrepancy of $H$ with respect to $(\mathbb{P},\Delta_\mathbb{P})$ is $\frac{1}{4}$.

Since
\[
\left(K_H+\frac{1}{2}\ell_B'+\ell_3'\right)\Big|_{\ell_3'}
=
K_{\ell_3'}+\frac{1}{2}\left(p_1+\cdots+p_n\right),
\]
the pair $(H,\frac{1}{2}\ell_B'+\ell_3')$ is purely log terminal. We obtain
\begin{equation}\label{eq:Extra-G2-1dot}
S_L(H)
=
\frac{1}{L^3}\int_0^\infty
\operatorname{vol}_{\mathbb P}(L-uH)\,du
=
\int_0^{\frac{21}{4(2n+1)}}
\left(1-\frac{4(2n+1)}{21}u\right)^3\,du
=
\frac{21}{16(2n+1)},
\end{equation}
\begin{equation}\label{eq:Extra-G2-2dot}
\begin{split}
S_L(V_{\bullet,\bullet}^{H};\ell_3')
&=
\frac{3}{L^3}
\int_0^{\frac{21}{4(2n+1)}}
\int_0^\infty
\operatorname{vol}_{H}
\left(\left.(L-uH)\right|_H-v\ell_3'\right)\,dv\,du\\
&=
\frac{3}{L^3}
\int_0^{\frac{21}{4(2n+1)}}
\int_0^{\frac{21-4(2n+1)u}{12(n-1)}}
\operatorname{vol}_{H}
\left(
\left(\frac{21}{4}-(2n+1)u\right)A|_H-(n-1)vA_H
\right)\,dv\,du\\
&=
\frac{3}{L^3}
\int_0^{\frac{21}{4(2n+1)}}
\int_0^{\frac{21-4(2n+1)u}{12(n-1)}}
\left(
\frac{7}{4}-\frac{2n+1}{3}u-(n-1)v
\right)^2 A_H^2\,dv\,du\\
&=
\frac{7}{16(n-1)},
\end{split}
\end{equation}
and 
\begin{equation}\label{eq:Extra-G2-3dot}
\begin{split}
S_L(W_{\bullet,\bullet,\bullet}^{H,\ell_3'};o_0)
&=
\frac{3}{L^3}
\int_0^{\frac{21}{4(2n+1)}}
\int_0^{\frac{21-4(2n+1)u}{12(n-1)}}
\left(
\left(
\left.(L-uH)\right|_H-v\ell_3'
\right)\cdot\ell_3'
\right)^2\,dv\,du\\
&=
A_H^2(n-1)^2
S_L(V_{\bullet,\bullet}^{H};\ell_3')\\
&=
\frac{7}{16}.
\end{split}
\end{equation}

By \cite[Corollary~4.8]{Fu23}, the values \eqref{eq:Extra-G2-1dot}, \eqref{eq:Extra-G2-2dot}, and \eqref{eq:Extra-G2-3dot} yield
\[
\delta_{o_0}(\mathbb{P},\Delta_\mathbb{P})
\geq
\min\left\{
\frac{4(2n+1)}{21},
\frac{16(n-1)}{21},
\frac{16}{7}
\right\}
\geq\frac{8}{7}.
\]

The automorphism group of the pair $\left(\mathbb{P},\frac12B+\frac34H\right)$ contains the subgroup $\mathfrak{G}$ generated by
\[
\begin{aligned}
[x_0:x_1:x_2:x_3]
&\longmapsto
[\mu_nx_0:x_1:x_2:x_3],\\
[x_0:x_1:x_2:x_3]
&\longmapsto
[x_0:\mu_nx_1:\mu_{3n}x_2:\mu_{3n}^{-1}x_3],
\end{aligned}
\]
where $\mu_r$ denotes a primitive $r$-th root of unity.

Every $\mathfrak{G}$-invariant point other than $o_0$ is contained in the surface $\Pi$. Moreover, every irreducible $\mathfrak{G}$-invariant subvariety of positive dimension intersects $\Pi$. Therefore, the inequality
$
\delta_p\left(\mathbb{P},\Delta_\mathbb{P}\right)\geq\frac{8}{7}
$
for every point $p\in\Pi\cup\{o_0\}$ implies that \[
\delta_\mathfrak{G}\left(\mathbb{P},\Delta_\mathbb{P}\right)\geq\frac{8}{7}>1.
\]
Therefore, by Theorem~\ref{theorem:G-delta}, the pair $(\mathbb{P},\Delta_\mathbb{P})$ is K-polystable, and consequently so is $Y$.

When  $n\equiv 1$  (mod $3$), i.e., $n=3k+1$ for some nonnegative integer $k$, we replace $\mathbb{P}$ at the beginning of the proof by $\mathbb{P}(1,1,2k+1,k)$. The argument then proceeds verbatim.
\end{proof}

\section{Proof of Main Theorem}

By Proposition~\ref{proposition:topology}, for every nonnegative integer $m$, 
the links of the hypersurface singularities $f_{m}$, $g_m$, $g_{1,m}$, and $g_{2,m}$ are homeomorphic to the following connected sums of~$S^3\times S^4$:
\[
\begin{array}{ll}
L(f_{m})\cong  (S^3\times S^4)^{\#12m},& 
L(g_m)\cong (S^3\times S^4)^{\#(12m+6)},\\
L(g_{1,m})\cong (S^3\times S^4)^{\#(6m+2)},&
L(g_{2, m})\cong  (S^3\times S^4)^{\#(6m+4)}.
\end{array}
\]
Therefore, to prove the Main Theorem, it suffices to show that each of these links admits a Sasaki--Einstein metric.

The existence of a Sasaki--Einstein metric on $L(f_{m})$  with $m\geq 1$ follows immediately from \cite[Theorem~1.4]{LST25}, since $f_{m}$ satisfies the required criterion for $m\geq 1$. Likewise, Theorem~\ref{theorem:g} implies that $L(g_m)$ admits a Sasaki--Einstein metric because the corresponding hypersurface $X$ admits a K\"ahler--Einstein metric.

Since $(S^3\times S^4)^{\#2}$ is already known to be able to host a Sasaki--Einstein structure (\cite[Table~2]{CL}), we may assume that $m\geq 1$ for $L(g_{1,m})$. Then, Theorem~\ref{theorem:h}, applied with $n=6m+3$ ($m\geq 1$) and $n=6m+5$ ($m\geq 0$), shows that the links $L(g_{1,m})$ and $L(g_{2,m})$ admit Sasaki--Einstein metrics, since the corresponding hypersurface $Y$ admits a K\"ahler--Einstein metric.

This completes the proof.

\bigskip
\textbf{Acknowledgements.}
Jeong and Park were supported by IBS-R003-D1 from the Institute for Basic Science in Korea. Kim was supported by the National Research Foundation of Korea under Grant number RS-2025-00513064, while Won was supported under Grant numbers RS-2025-00513064 and RS-2026-25506097. ChatGPT was used to assist the authors in discovering the polynomials $g_{1,m}$ and $g_{2,m}$  in \eqref{TS-type} that satisfy the properties established in Proposition~\ref{proposition:topology}.

\end{document}